\documentclass[default]{sn-jnl}

\jyear{2021}%

\usepackage{etoolbox}   
\makeatletter
\patchcmd{\@maketitle}{\artauthors}{\centerline{\artauthors}}{}{}
\makeatother

\makeatletter
\patchcmd{\ps@headings}
{\hbox to \hsize{\hfill Springer Nature 2021 \LaTeX\ template\hfill}}
{\hbox to \hsize{}}
{}
{}
\patchcmd{\ps@headings}
{\hbox to \hsize{\hfill Springer Nature 2021 \LaTeX\ template\hfill}}
{\hbox to \hsize{}}
{}
{}
\patchcmd{\ps@titlepage}
{\hbox to \hsize{\hfill Springer Nature 2021 \LaTeX\ template\hfill}}
{\hbox to \hsize{}}
{}
{}
\makeatother

\theoremstyle{thmstyleone}%
\newtheorem{theorem}{Theorem}
\newtheorem{corollary}{Corollary}
\newtheorem{conjecture}{Conjecture}
\newtheorem{lemma}{Lemma}
\newtheorem{claim}{Claim}
\newtheorem{case}{Case}

\theoremstyle{thmstyletwo}%

\theoremstyle{thmstylethree}%

\begin{document}

\title[Matching-star size Ramsey numbers under connectivity constraint]{Matching-star size Ramsey numbers under connectivity constraint}


\author[1]{\fnm{Fanghua} \sur{Guo}}\email{fhguo@stu.hebtu.edu.cn}

\author*[1,2]{\fnm{Yanbo} \sur{Zhang}}\email{ybzhang@hebtu.edu.cn}

\author[3]{\fnm{Yunqing} \sur{Zhang}}\email{yunqingzh@nju.edu.cn}

\affil[1]{\orgdiv{School of Mathematical Sciences}, \orgname{Hebei Normal University}, \orgaddress{\city{Shijiazhuang}, \postcode{050024}, \state{Hebei}, \country{China}}}

\affil[2]{\orgname{Hebei International Joint Research Center for Mathematics and Interdisciplinary Science}, \orgaddress{\city{Shijiazhuang}, \postcode{050024}, \state{Hebei}, \country{China}}}

\affil[3]{\orgdiv{Department of Mathematics}, \orgname{Nanjing University}, \orgaddress{\city{Nanjing}, \postcode{210093}, \state{Jiangsu}, \country{China}}}

\abstract{Recently, Caro, Patk\'os, and Tuza (2022) introduced the concept of connected Tur\'an number. We study a similar parameter in Ramsey theory. Given two graphs $G_1$ and $G_2$, the size Ramsey number $\hat{r}(G_1,G_2)$ refers to the smallest number of edges in a graph $G$ such that for any red-blue edge-coloring of $G$, either a red subgraph $G_1$ or a blue subgraph $G_2$ is present in $G$. If we further restrict the host graph $G$ to be connected, we obtain the connected size Ramsey number, denoted as $\hat{r}_c(G_1,G_2)$. Erd\H{o}s and Faudree (1984) proved that $\hat r(nK_2,K_{1,m})=mn$ for all positive integers $m,n$. In this paper, we concentrate on the connected analog of this result. Rahadjeng, Baskoro, and Assiyatun (2016) provided the exact values of $\hat r_c(nK_2,K_{1,m})$ for $n=2,3$. We establish a more general result: for all positive integers $m$ and $n$ with $m\ge (n^2+2pn+n-3)/2$, we have $\hat r_c(nK_{1,p},K_{1,m})=n(m+p)-1$. As a corollary, $\hat r_c(nK_2,K_{1,m})=nm+n-1$ for $m\ge (n^2+3n-3)/2$. We also propose a conjecture for the interested reader.}




\maketitle

\section{Introduction}\label{sec1}

Given two graphs $G_1$ and $G_2$, we define $G\to (G_1,G_2)$ if, for any red-blue edge-coloring of the graph $G$, $G$ contains either a red subgraph isomorphic to $G_1$, or a blue subgraph isomorphic to $G_2$. The term ``size'' commonly refers to the number of edges in a graph. The concept of the \emph{size Ramsey number} $\hat{r}(G_1,G_2)$ was introduced by Erd\H os, Faudree, Rousseau, and Schelp~\cite{Erdoes1978size} in 1978. It represents the smallest positive integer $m$ for which there exists a graph $G$ of size $m$ satisfying $G\to (G_1,G_2)$. Size Ramsey numbers are one of the most well-established topics in graph Ramsey theory, and a comprehensive survey can be found in~\cite{Faudree2002survey}.

Determining size Ramsey numbers presents a more challenging task compared to (order) Ramsey numbers. A key focus in this field is to find the graphs $G$ for which $\hat{r}(G,G)$ exhibits linear growth with respect to $\vert V(G)\vert$. Extensive research has been conducted on this topic if $G$ is a path~\cite{Beck1983size,Dudek2017some}, a tree with bounded maximum degree~\cite{Dellamonica2012size,Haxell1995size}, a cycle~\cite{Haxell1995Induced,Javadi2019Size}, and others. On the other hand, numerous studies have aimed to determine the exact values of size Ramsey numbers, as evidenced by a variety of papers addressing this topic~\cite{Burr1978Ramsey,Davoodi2021Conjecture,Erdoes1984Size,Erdoes1978size,Faudree1983Size,Faudree1983Sizea,Harary1983Generalized,Javadi2018Question,Lortz1998Size,Lortz2021Size,Miralaei2019Size,Sheehan1984Class}. Among them, Erd\H{o}s and Faudree~\cite{Erdoes1984Size} proved that $\hat r(nK_2,K_{1,m})=mn$ for all positive integers $m,n$.

Several variants of the size Ramsey number have also been studied. In 2015, Rahadjeng, Baskoro, and Assiyatun~\cite{rahadjeng2015connected} investigated a specific variant called the connected size Ramsey number. This variant adds the condition that the host graph $G$ must be connected. More precisely, the \emph{connected size Ramsey number} $\hat{r}_c(G_1,G_2)$ represents the minimum number of edges required for a connected graph $G$ such that $G\to (G_1,G_2)$. It is evident that $\hat{r}(G_1,G_2)\leq \hat{r}_c(G_1,G_2)$, and equality holds when both $G_1$ and $G_2$ are connected graphs. However, the latter parameter presents challenges when either $G_1$ or $G_2$ is disconnected. Existing research primarily focuses on the connected size Ramsey numbers involving a matching versus a sparse graph, such as a path, a star, or a cycle.

In the context of the connected size Ramsey number of a matching versus a star, Rahadjeng, Baskoro, and Assiyatun~\cite{Rahadjeng2016Connected} derived an upper bound for $\hat r_c(nK_2,K_{1,m})$, which is $nm+n-1$. Here, a matching $nK_2$ consists of $n$ pairwise non-adjacent edges, and a star $K_{1,m}$ is a tree with $m$ edges sharing a common endpoint. They also demonstrated that this bound is tight when $n=2$ and $3$. On the other hand, when $m$ is small, this upper bound cannot be tight. For instance, Wang, Song, Zhang, and Zhang~\cite{Wang2022Connected} showed that $\hat r_c(nK_2,K_{1,2})=\left\lceil (5n-1)/2\right\rceil$. In this paper, we generalize the results of Rahadjeng, Baskoro, and Assiyatun, and determine the exact values of $\hat r_c(nK_{1,p},K_{1,m})$ for larger values of $m$, as stated by the following theorem. Here, $nK_{1,p}$ is the disjoint union of $n$ copies of $K_{1,p}$, which we call a star matching.

\begin{theorem}
	\label{thm:mainresult}
	For positive integers $m, n$, and $p$ with $m\ge (n^2+2pn+n-3)/2$, we have \[\hat r_c(nK_{1,p},K_{1,m})=n(m+p)-1.\]
\end{theorem}

When $p=1$, $nK_{1,p}$ is actually a matching with $n$ edges. We have the following corollary.
\begin{corollary}
	For positive integers $m$ and $n$ with $m\ge (n^2+3n-3)/2$, we have \[\hat r_c(nK_2,K_{1,m})=nm+n-1.\]
\end{corollary}

It would be intriguing to enhance the lower bound of the parameter $m$ further. Specifically, we propose the following conjecture.

\begin{conjecture}
	There exists a positive constant $c$ such that for all positive integers $m$, $n$, and $p$ with $m\geq c(n+p)$, we have \[\hat r_c(nK_{1,p},K_{1,m})=n(m+p)-1.\]
\end{conjecture}

This conjecture is beyond our ability, and we leave it to interested readers.

\section{Preliminaries and methods}\label{sec2}

A $(G_1,G_2)$-coloring of $G$ is a red-blue edge-coloring of $G$ such that there is neither a red $G_1$ nor a blue $G_2$. Therefore, the statement ``$G\not\to (G_1,G_2)$'' is equivalent to saying that there exists a $(G_1,G_2)$-coloring for the graph $G$. Throughout the discussion, we will use these two equivalent statements interchangeably multiple times. For all other terminologies, please refer to~\cite{Bondy2008Graph}.

The following folklore is needed to show the lower bound.
\begin{lemma}\label{lemma}
	For any matching $M$ and any vertex cover $S$ of a graph $G$, we have $\vert M\vert\le \vert S\vert$.
\end{lemma}

In the second half of this section, let us briefly summarize the proof techniques used in the paper. The main part of the proof lies in determining the lower bound of $\hat{r}_c(nK_{1,p},K_{1,m})$. We will induct on $n$ and employ the method of minimal counterexample. Suppose that a connected graph $G$ has at most $n(m+p)-2$ edges and satisfies $G\to(nK_{1,p},K_{1,m})$. Among all graphs $G$ that satisfy the condition, we choose one with the smallest number of edges. We partition the set $V(G)$ into two sets based on their degrees, and then discover some properties of $G$. We shall find at least $m+p$ edges in $G$ such that when these edges are removed, the graph induced by the remaining edges, denoted as $H$, is still connected. According to the induction hypothesis, $H\not\to((n-1)K_{1,p},K_{1,m})$. Building upon an $((n-1)K_{1,p},K_{1,m})$-coloring of $H$, we then color the edges in $E(G)\setminus E(H)$, and obtain an $(nK_{1,p},K_{1,m})$-coloring of $G$, leading to a contradiction.

\section{Proof of the main theorem}\label{sec3}

For any red-blue edge-coloring of $K_{1,m+p-1}$, there is either a red copy of $K_{1,p}$ or a blue copy of $K_{1,m}$. Consequently, it is easy to check that $nK_{1,m+p-1}\to (nK_{1,p},K_{1,m})$. Because $nK_{1,m+p-1}$ has $n$ connected components, we can make this graph connected by adding $n-1$ edges. Let us call this graph $G$. It has $n(m+p)-1$ edges and $G\to(nK_{1,p},K_{1,m})$. Hence the upper bound follows.

For the lower bound, let $G$ be an arbitrary connected graph with $n(m+p)-2$ edges. We aim to prove that $G\not\to (nK_{1,p},K_{1,m})$, where $m \geq (n^2+2pn+n-3)/2$. By induction on $n$. The base case $n=1$ is straightforward. For each $k\in [n-1]$, assume that any connected graph $G$ with $k(m+p)-2$ edges satisfies $G\not\to (kK_{1,p},K_{1,m})$. We will now establish the case for $k=n$. Throughout the proof, $U$ denotes the set of vertices with degree at least $m$ in $G$. Set $t=\vert U\vert $ and $V=V(G)\setminus U$.

We will structure our discussion according to the value of $t$, which naturally leads to three cases in the proof. The first two cases are relatively straightforward to analyze, while our discussion focuses on the last case.

\begin{case}
$t\le n-1$.
\end{case}

We color the edges incident to $U$ red and the remaining edges blue. Thus, $U$ forms a vertex cover of the red edges. It is worth noting that the maximum degree induced by the blue edges does not exceed $m-1$. Therefore, this coloring is an $(nK_{1,p},K_{1,m})$-coloring of $G$.

\begin{case}
	$t\ge n+1$.
\end{case}

The graph $G$ has size at least $tm-\binom{t}{2}$. This is because the number of edges with at least one end in $U$ is at least $tm-\binom{t}{2}$. Consider the function $f(t)=tm-\tbinom{t}{2}=-\frac{1}{2}{t}^{2}+(m+\frac{1}{2})t$, which the axis of symmetry is $t=m+1/2$. Note that $t\le 2n$, since if $t\ge 2n+1$ then $G$ has at least $(2n+1)m/2$ edges, contradicting the assumption that $G$ has $nm+n-2$ edges. Since $n+1\le 2n<m+1/2$, the function is minimized by the property of quadratic functions when $t=n+1$. Hence we have
		\begin{align*}	
           f(t)&\ge f(n+1)\\
		   &=(n+1)m-\frac{(n+1)n}{2}\\
           &\ge nm+\frac{{n}^{2}+2pn+n-3}{2}-\frac{(n+1)n}{2}\\
           &=nm+\frac{2pn-3}{2}\\
           &>n(m+p)-2.\\
		\end{align*}
This is a contradiction since $G$ does not have enough edges.

\begin{case}
	$t=n$.
\end{case}

In this case, we use proof by contradiction. Suppose there exists a connected graph $G$ with at most $n(m+p)-2$ edges such that $G \to (nK_{1,p},K_{1,m})$. We select $G$ to be a connected graph with the minimum number of edges that satisfies this property. In other words, for any connected graph $H$ with $e(H)<e(G)$, we always have $H \not\to (nK_{1,p},K_{1,m})$. We will characterize the properties of $G$ through some claims until we finally prove that such a $G$ does not exist.

	\begin{claim}\label{claim3}
       $G[V]$ is an empty graph.
	\end{claim}
\begin{proof}
Recall that $V=V(G)\setminus U$. Assume that there is an edge $e$ in $G[V]$ such that $G-e$ is still connected. According to the selection of $G$ at the beginning of this case, $G-e\not\to(nK_{1,p},K_{1,m})$. Hence $G-e$ has an $(nK_{1,p},K_{1,m})$-coloring. We then color $e$ blue, which does not contribute to a blue star $K_{1,m}$ since the center of any $K_{1,m}$ must be in $U$. It follows that $G\not\to(nK_{1,p},K_{1,m})$, a contradiction.

Assume that there is an edge $e$ in $G[V]$ such that $G$ is disconnected after deleting $e$. Then $G-e$ has two components, written as $G_1$ and $G_2$. There exist two positive integers $a_1$ and $a_2$ such that 
\[(m+p)a_1-1\le \vert E(G_1)\vert \le (m+p)(a_1+1)-2,\]
\[(m+p)a_2-1\le \vert E(G_2)\vert \le (m+p)(a_2+1)-2.\]
Observe that both $a_1$ and $a_2$ are less than $n$. We have $G_1\not\to((a_1+1)K_{1,p},K_{1,m})$ and $G_2\not\to((a_2+1)K_{1,p},K_{1,m})$ by induction. That is, there exists an $((a_1+1)K_{1,p},K_{1,m})$-coloring in $G_1$ and an $((a_2+1)K_{1,p},K_{1,m})$-coloring in $G_2$. According to the coloring, there are at most $a_1+a_2$ copies of red $K_{1,p}$ in $G-e$, and there is no blue $K_{1,m}$. By adding the two inequalities, we obtain that $(m+p)(a_1+a_2)-2\le \vert E(G_1)\vert +\vert E(G_2)\vert$, which together with the fact $\vert E(G_1)\vert +\vert E(G_2)\vert \le (m+p)n-3$ implies $a_1+a_2<n$. Based on the colorings of $G_1$ and $G_2$, we color $e$ blue. This is an $(nK_{1,p},K_{1,m})$-coloring and hence $G\not\to (nK_{1,p},K_{1,m})$, a contradiciton.
\end{proof}
	\begin{claim}
       $G[U]$ is an empty graph.
	\end{claim}
\begin{proof}
The proof of this claim is quite similar to the proof of the previous one. Assuming that an edge $e=u_1u_2$ exists in $G[U]$, we divide the discussion into two cases: when $G-e$ is connected and when $G-e$ is disconnected.

First, we assume that $G-e$ is connected. According to the assumption of $G$ at the beginning of this case, $G-e\not\to (nK_{1,p},K_{1,m})$. Based on an $(nK_{1,p},K_{1,m})$-coloring of $G-e$, we proceed to color $e$ red and obtain a coloring of $G$. If $G$ contains a red $nK_{1,p}$, it must include the edge $e$, since otherwise $G-e$ would also contain this red $nK_{1,p}$. Due to Claim~\ref{claim3}, $U\setminus \{ u_1,u_2 \}$ is a vertex cover of a red $(n-1)K_{1,p}$. However, combining Lemma~\ref{lemma} and the fact that $\vert U\setminus \{ u_1,u_2 \}\vert =n-2$, $U\setminus \{ u_1,u_2 \}$ cannot be a vertex cover for $(n-1)K_{1,p}$, a contradiction. Thus, this coloring is an $(nK_{1,p},K_{1,m})$-coloring of $G$, which contradicts the assumption that $G\to(nK_{1,p},K_{1,m})$.

Second, we assume that $G-e$ is disconnected. Then $G-e$ consists of two components, denoted as $G_1$ and $G_2$. There exist two positive integers $a_1$ and $a_2$ such that 
\[(m+p)a_1-1\le \vert E(G_1)\vert \le (m+p)(a_1+1)-2\]\[(m+p)a_2-1\le \vert E(G_2)\vert \le (m+p)(a_2+1)-2\]
Note that both $a_1$ and $a_2$ are less than $n$. By induction, we have $G_1\not\to((a_1+1)K_{1,p},K_{1,m})$ and $G_2\not\to((a_2+1)K_{1,p},K_{1,m})$. Consequently, we can find an $((a_1+1)K_{1,p},K_{1,m})$-coloring in $G_1$ and an $((a_2+1)K_{1,p},K_{1,m})$-coloring in $G_2$. With this coloring, there are at most $a_1+a_2$ copies of red $K_{1,p}$ in $G-e$, and no blue $K_{1,m}$ is present. By adding the two inequalities, we deduce $(m+p)(a_1+a_2)-2\le \vert E(G_1)\vert +\vert E(G_2)\vert$. Since $\vert E(G_1)\vert +\vert E(G_2)\vert \le (m+p)n-3$, we have $a_1+a_2<n$. Based on the coloring of $G_1$ and $G_2$, we color $e$ red. If $G$ has a red $nK_{1,p}$, it must include the edge $e$, since otherwise $G-e$ would also contain this $K_{1,p}$. By Claim~\ref{claim3}, $U\setminus \{ u_1,u_2 \}$ is a vertex cover of $(n-1)K_{1,p}$. However, combining Lemma~\ref{lemma} and the fact that $\vert U\setminus \{ u_1,u_2 \}\vert =n-2$, $U\setminus \{ u_1,u_2 \}$ cannot be a vertex cover for $(n-1)K_{1,p}$, which leads to a contradiction. Thus, this coloring is an $(nK_{1,p},K_{1,m})$-coloring of $G$, contradicting the assumption that $G\to(nK_{1,p},K_{1,m})$.
\end{proof}

We see that every edge of $G$ has one endpoint in $U$ and the other in $V$. We have the following claim for $p\ge 2$.

\begin{claim}\label{centerinU}
	When $p\ge 2$, if $G$ contains a red $nK_{1,p}$ as a subgraph, then each vertex of $U$ is a center of the red $K_{1,p}$'s.
 \end{claim}

\begin{proof}
	Proof by contradiction. Let us assume that $G$ contains a red $nK_{1,p}$ as a subgraph, and suppose that some vertex of $U$ is not a center of the $K_{1,p}$'s. Then there exists one of the $K_{1,p}$'s, denoted as $F$, with its center located in $V$. According to Claim~\ref{claim3}, all the leaves of $F$ are in $U$. There are $n-p$ vertices in $U\setminus V(F)$. According to the two previous claims, each edge of the remaining $n-1$ copies of the red $K_{1,p}$ has one end in $U\setminus V(F)$. Consequently, the $n-p$ vertices form a vertex cover for $(n-1)K_{1,p}$, which is impossible. This validates our claim.
\end{proof}

In the following, we construct two new graphs, $G'$ and $G''$. For $u\in U$ and $v\in V$, if the graph $G$ remains connected after deleting the edge $uv$ while ensuring $u,v \in V(G)$, then we delete the edge $uv$ from $G$ and add a new edge $uv'$, where $v'$ is a new vertex. We continue performing this operation on the edges of $G$ until it can no longer be done. As a result, we have constructed a new graph, denoted by $G'$. For every edge $e=uv$, where $u\in U$ and $v\in V$, according to the above construction, $G'-e$ must be disconnected. Recall that a graph is a tree if and only if it is minimally connected. Thus, $G'$ is a tree.

It is easy to verify from the definition of $U$ and Claim~\ref{centerinU} that if $G'$ has an $(nK_{1,p},K_{1,m})$-coloring, then $G$ has a corresponding $(nK_{1,p},K_{1,m})$-coloring. Based on the assumption we made at the beginning of this case, namely $G\to(nK_{1,p}, K_{1,m})$, we can conclude that \[G'\to(nK_{1,p}, K_{1,m}).\] We still partition $V(G')$ into two parts, $U$ and $V$. If a vertex in $V(G')$ belongs to the set $U$ in the original graph $G$, we still assign it to $U$. Otherwise, we assign it to $V$.

Now we delete all the leaves from $G'$, and the resulting new graph is denoted as $G''$.
We select two vertices (not necessarily distinct), $v_1$ and $v_2$, from $V$ in such a way that the distance between $v_1$ and $v_2$ on the tree $G''$ is maximized. In other words, $d_{G''}(v_1,v_2)=\max \{ d_{G''}(v,v') \mid v,v'\in V \}$. From this, we deduce that the vertex $v_2$ has at least one neighbor that is a leaf. Moreover, due to the maximality of $d_{G''}(v_1,v_2)$, if $v_2$ has exactly one neighbor that is a leaf, then the degree of $v_2$ is two.

There are two cases in the graph $G''$. The first case is when $v_2$ has at least two neighbors that are leaves, denoted as $u_1$ and $u_2$. The second case is when $v_2$ has exactly one neighbor that is a leaf, denoted as $u_3$, and the degree of $v_2$ is two. Regarding the degrees of three vertices $u_1,u_2,u_3$ in the graph $G'$, the following claim holds.

\begin{claim}\label{degree}
	$d_{G'}(u_i)\ge m+p-1$ for $i \in [3]$. Moreover, either $d_{G'}(u_1)\ge m+p$ or $d_{G'}(u_2)\ge m+p$.
 \end{claim}

 \begin{proof}
	Suppose there exists $i \in [3]$ such that $d_{G'}(u_i)\le m+p-2$. In the graph $G'$, if $u_i$ has at least $m-1$ leaves, we color the edges between $u_i$ and exactly $m-1$ leaves as blue, and all other edges in $G'$ as red. On the other hand, if $u_i$ has at most $m-2$ leaves, we color all edges incident to $u_i$ as blue, and all other edges in $G'$ as red. This graph has at most $m-1$ blue edges, hence it does not contain a blue $K_{1,m}$. Since $u_i$ is incident to at most $p-1$ red edges, $u_i$ cannot be the center of a red $K_{1,p}$. According to Claim~\ref{centerinU}, $G'$ does not contain a red $nK_{1,p}$ for $p\ge 2$. It is easy to verify that the same conclusion holds for $p=1$. Therefore, $G'\not\to(nK_{1,p},K_{1,m})$. This contradiction demonstrates that $d_{G'}(u_i)\ge m+p-1$ for $i \in [3]$.
	
	Suppose $d_{G'}(u_1)=d_{G'}(u_2)=m+p-1$. In the graph $G'$, we select $m-1$ leaves of $u_1$ and color the edges between $u_1$ and these $m-1$ leaves as blue. Similarly, we select $m-1$ leaves of $u_2$ and color the edges between $u_2$ and these $m-1$ leaves as blue. Then, we color all other edges in $G'$ as red. It is noteworthy that there are exactly $p$ red edges incident to $u_1$ and $p$ red edges incident to $u_2$. However, since both $u_1v_2$ and $u_2v_2$ are red edges, and $nK_{1,p}$ represents the disjoint union of $n$ copies of $K_{1,p}$, it follows that $u_1$ and $u_2$ cannot be the centers of different $K_{1,p}$ components in $nK_{1,p}$. According to Claim~\ref{centerinU}, $G'$ does not contain a red $nK_{1,p}$ for $p\ge 2$. The same conclusion holds for $p=1$. Hence, we have $G'\not\to(nK_{1,p},K_{1,m})$, which contradicts our assumption. Consequently, either $d_{G'}(u_1)\ge m+p$ or $d_{G'}(u_2)\ge m+p$.
 \end{proof}

If there exists $i\in [3]$ such that $d_{G'}(u_i)\geq m+p$, without loss of generality, assume that $d_{G'}(u_1)\ge m+p$. We delete the vertex $u_1$ and its adjacent leaves from $G'$, resulting in a connected graph. By the induction hypothesis, $G'\setminus \{ u_1 \}\not\to ((n-1)K_{1,p}, K_{1,m})$. Building upon the $((n-1)K_{1,p}, K_{1,m})$-coloring of $G'\setminus \{ u_1 \}$, we color all edges incident to $u_1$ in red. This gives us an $(nK_{1,p}, K_{1,m})$-coloring of $G'$, which is a contradiction.

According to Claim~\ref{degree}, only one case remains: $d_{G'}(u_3)=m+p-1$, $v_2$ is adjacent to $u_3$, and $v_2$ has degree two in $G'$. We remove $v_2$, $u_3$, and the adjacent leaves of $u_3$ from $G'$. In this way, the resulting graph remains connected. Note that we have removed $m+p$ edges during this process. By the induction hypothesis, $G'\setminus \{ v_2, u_3 \}\not\to ((n-1)K_{1,p}, K_{1,m})$. Expanding upon the $((n-1)K_{1,p}, K_{1,m})$-coloring of $G'\setminus \{ v_2, u_3 \}$, we establish an $(nK_{1,p}, K_{1,m})$-coloring of $G'$ as follows: we select $m-1$ leaves of $u_3$ and color the edges between $u_3$ and these $m-1$ leaves as blue. Subsequently, we assign red color to the remaining edges connected to $u_3$ and $v_2$. Since $u_3$ is precisely incident to $p$ red edges, if $u_3$ is the center of a red $K_{1,p}$, then $v_2$ must also be a vertex of this $K_{1,p}$ and cannot belong to any other red $K_{1,p}$. This results in an $(nK_{1,p}, K_{1,m})$-coloring of $G'$, which contradicts our assumption. This completes the proof of the main theorem.

\bmhead{Acknowledgments}

The second author was partially supported by NSFC under grant numbers 11601527, while
the third author was partially supported by NSFC under grant numbers 12161141003 and 12271246.





\begin{thebibliography}{21}
	\ifx \bisbn   \undefined \def \bisbn  #1{ISBN #1}\fi
	\ifx \binits  \undefined \def \binits#1{#1}\fi
	\ifx \bauthor  \undefined \def \bauthor#1{#1}\fi
	\ifx \batitle  \undefined \def \batitle#1{#1}\fi
	\ifx \bjtitle  \undefined \def \bjtitle#1{#1}\fi
	\ifx \bvolume  \undefined \def \bvolume#1{\textbf{#1}}\fi
	\ifx \byear  \undefined \def \byear#1{#1}\fi
	\ifx \bissue  \undefined \def \bissue#1{#1}\fi
	\ifx \bfpage  \undefined \def \bfpage#1{#1}\fi
	\ifx \blpage  \undefined \def \blpage #1{#1}\fi
	\ifx \burl  \undefined \def \burl#1{\textsf{#1}}\fi
	\ifx \doiurl  \undefined \def \doiurl#1{\url{https://doi.org/#1}}\fi
	\ifx \betal  \undefined \def \betal{\textit{et al.}}\fi
	\ifx \binstitute  \undefined \def \binstitute#1{#1}\fi
	\ifx \binstitutionaled  \undefined \def \binstitutionaled#1{#1}\fi
	\ifx \bctitle  \undefined \def \bctitle#1{#1}\fi
	\ifx \beditor  \undefined \def \beditor#1{#1}\fi
	\ifx \bpublisher  \undefined \def \bpublisher#1{#1}\fi
	\ifx \bbtitle  \undefined \def \bbtitle#1{#1}\fi
	\ifx \bedition  \undefined \def \bedition#1{#1}\fi
	\ifx \bseriesno  \undefined \def \bseriesno#1{#1}\fi
	\ifx \blocation  \undefined \def \blocation#1{#1}\fi
	\ifx \bsertitle  \undefined \def \bsertitle#1{#1}\fi
	\ifx \bsnm \undefined \def \bsnm#1{#1}\fi
	\ifx \bsuffix \undefined \def \bsuffix#1{#1}\fi
	\ifx \bparticle \undefined \def \bparticle#1{#1}\fi
	\ifx \barticle \undefined \def \barticle#1{#1}\fi
	\bibcommenthead
	\ifx \bconfdate \undefined \def \bconfdate #1{#1}\fi
	\ifx \botherref \undefined \def \botherref #1{#1}\fi
	\ifx \url \undefined \def \url#1{\textsf{#1}}\fi
	\ifx \bchapter \undefined \def \bchapter#1{#1}\fi
	\ifx \bbook \undefined \def \bbook#1{#1}\fi
	\ifx \bcomment \undefined \def \bcomment#1{#1}\fi
	\ifx \oauthor \undefined \def \oauthor#1{#1}\fi
	\ifx \citeauthoryear \undefined \def \citeauthoryear#1{#1}\fi
	\ifx \endbibitem  \undefined \def \endbibitem {}\fi
	\ifx \bconflocation  \undefined \def \bconflocation#1{#1}\fi
	\ifx \arxivurl  \undefined \def \arxivurl#1{\textsf{#1}}\fi
	\csname PreBibitemsHook\endcsname
	
	\bibitem{Beck1983size}
	\begin{barticle}
		\bauthor{\bsnm{Beck}, \binits{J.}}:
		\batitle{On size {Ramsey} number of paths, trees, and circuits. {I}}.
		\bjtitle{J. Graph Theory}
		\bvolume{7}(\bissue{1}),
		\bfpage{115}--\blpage{129}
		(\byear{1983}).
		\doiurl{10.1002/jgt.3190070115}
	\end{barticle}
	\endbibitem
	
	\bibitem{Bondy2008Graph}
	\begin{bbook}
		\bauthor{\bsnm{Bondy}, \binits{J.A.}},
		\bauthor{\bsnm{Murty}, \binits{U.S.R.}}:
		\bbtitle{Graph Theory}.
		\bpublisher{Springer}, (\byear{2008}).
		\doiurl{10.1007/978-1-84628-970-5}
	\end{bbook}
	\endbibitem
	
	\bibitem{Burr1978Ramsey}
	\begin{barticle}
		\bauthor{\bsnm{Burr}, \binits{S.A.}},
		\bauthor{\bsnm{Erd{\H{o}}s}, \binits{P.}},
		\bauthor{\bsnm{Faudree}, \binits{R.J.}},
		\bauthor{\bsnm{Rousseau}, \binits{C.C.}},
		\bauthor{\bsnm{Schelp}, \binits{R.H.}}:
		\batitle{Ramsey-minimal graphs for multiple copies}.
		\bjtitle{Indagationes Mathematicae (Proceedings)}
		\bvolume{81}(\bissue{2}),
		\bfpage{187}--\blpage{195}
		(\byear{1978}).
		\doiurl{10.1016/1385-7258(78)90036-7}
	\end{barticle}
	\endbibitem

	\bibitem{Caro2022Connected}
	\begin{barticle}
		\bauthor{\bsnm{Caro}, \binits{Y.}},
		\bauthor{\bsnm{Patk\'os}, \binits{B.}},
		\bauthor{\bsnm{Tuza}, \binits{Z.}}:
		\batitle{Connected Tur\'an number of trees}.
		\bjtitle{arXiv preprint, arXiv:2208.06126}
		(\byear{2022}).
		{\href{https://arxiv.org/abs/2208.06126}{{arXiv:2208.06126}}}
	\end{barticle}
	\endbibitem
	
	\bibitem{Davoodi2021Conjecture}
	\begin{barticle}
		\bauthor{\bsnm{Davoodi}, \binits{A.}},
		\bauthor{\bsnm{Javadi}, \binits{R.}},
		\bauthor{\bsnm{Kamranian}, \binits{A.}},
		\bauthor{\bsnm{Raeisi}, \binits{G.}}:
		\batitle{On a conjecture of Erd{\H{o}}s on size {R}amsey number of star forests}.
		\bjtitle{arXiv preprint, arXiv:2111.02065}
		(\byear{2021}).
		{\href{https://arxiv.org/abs/2111.02065}{{arXiv:2111.02065}}}
	\end{barticle}
	\endbibitem
	
	\bibitem{Dellamonica2012size}
	\begin{barticle}
		\bauthor{\bsnm{{Dellamonica Jr.}}, \binits{D.}}:
		\batitle{The size-{Ramsey} number of trees}.
		\bjtitle{Random Struct. Algor.}
		\bvolume{40}(\bissue{1}),
		\bfpage{49}--\blpage{73}
		(\byear{2012}).
		\doiurl{10.1002/rsa.20363}
	\end{barticle}
	\endbibitem
	
	\bibitem{Dudek2017some}
	\begin{barticle}
		\bauthor{\bsnm{Dudek}, \binits{A.}},
		\bauthor{\bsnm{Pra{\l}at}, \binits{P.}}:
		\batitle{On some multicolor {Ramsey} properties of random graphs}.
		\bjtitle{SIAM J. Discrete Math.}
		\bvolume{31}(\bissue{3}),
		\bfpage{2079}--\blpage{2092}
		(\byear{2017}).
		\doiurl{10.1137/16m1069717}
	\end{barticle}
	\endbibitem
	
	\bibitem{Erdoes1978size}
	\begin{barticle}
		\bauthor{\bsnm{Erd{\H{o}}s}, \binits{P.}},
		\bauthor{\bsnm{Faudree}, \binits{R.J.}},
		\bauthor{\bsnm{Rousseau}, \binits{C.C.}},
		\bauthor{\bsnm{Schelp}, \binits{R.H.}}:
		\batitle{The size {Ramsey} number}.
		\bjtitle{Period. Math. Hungar.}
		\bvolume{9}(\bissue{1-2}),
		\bfpage{145}--\blpage{161}
		(\byear{1978}).
		\doiurl{10.1007/bf02018930}
	\end{barticle}
	\endbibitem
	
	\bibitem{Erdoes1984Size}
	\begin{bchapter}
		\bauthor{\bsnm{Erd{\H{o}}s}, \binits{P.}},
		\bauthor{\bsnm{Faudree}, \binits{R.J.}}:
		\bctitle{Size {Ramsey} numbers involving matchings}.
		In: \bbtitle{Finite and Infinite Sets},
		pp. \bfpage{247}--\blpage{264}.
		\bpublisher{Elsevier},
		\blocation{Amsterdam, The Netherlands}
		(\byear{1984}).
		\doiurl{10.1016/b978-0-444-86893-0.50019-x}
	\end{bchapter}
	\endbibitem
	
	\bibitem{Faudree2002survey}
	\begin{bchapter}
		\bauthor{\bsnm{Faudree}, \binits{R.J.}},
		\bauthor{\bsnm{Schelp}, \binits{R.H.}}:
		\bctitle{A survey of results on the size {R}amsey number}.
		In: \bbtitle{Paul Erd{\H{o}}s and His Mathematics, II},
		pp. \bfpage{291}--\blpage{309}.
		\bpublisher{Springer}, \blocation{Berlin}
		(\byear{2002})
	\end{bchapter}
	\endbibitem
	
	\bibitem{Faudree1983Size}
	\begin{barticle}
		\bauthor{\bsnm{Faudree}, \binits{R.J.}},
		\bauthor{\bsnm{Sheehan}, \binits{J.}}:
		\batitle{Size {Ramsey} numbers for small-order graphs}.
		\bjtitle{J. Graph Theory}
		\bvolume{7}(\bissue{1}),
		\bfpage{53}--\blpage{55}
		(\byear{1983}).
		\doiurl{10.1002/jgt.3190070107}
	\end{barticle}
	\endbibitem
	
	\bibitem{Faudree1983Sizea}
	\begin{barticle}
		\bauthor{\bsnm{Faudree}, \binits{R.J.}},
		\bauthor{\bsnm{Sheehan}, \binits{J.}}:
		\batitle{Size {Ramsey} numbers involving stars}.
		\bjtitle{Discrete Math.}
		\bvolume{46}(\bissue{2}),
		\bfpage{151}--\blpage{157}
		(\byear{1983}).
		\doiurl{10.1016/0012-365x(83)90248-0}
	\end{barticle}
	\endbibitem
	
	\bibitem{Harary1983Generalized}
	\begin{bchapter}
		\bauthor{\bsnm{Harary}, \binits{F.}},
		\bauthor{\bsnm{Miller}, \binits{Z.}}:
		\bctitle{Generalized {Ramsey} theory {VIII}. the size {Ramsey} number of small
			graphs}.
		In: \bbtitle{Studies in Pure Mathematics},
		pp. \bfpage{271}--\blpage{283}.
		\bpublisher{Birkh{\"a}user},
		\blocation{Basel, Switzerland}
		(\byear{1983}).
		\doiurl{10.1007/978-3-0348-5438-2_25}
	\end{bchapter}
	\endbibitem
	
	\bibitem{Haxell1995size}
	\begin{barticle}
		\bauthor{\bsnm{Haxell}, \binits{P.E.}},
		\bauthor{\bsnm{Kohayakawa}, \binits{Y.}}:
		\batitle{The size-{Ramsey} number of trees}.
		\bjtitle{Israel J. Math.}
		\bvolume{89}(\bissue{1-3}),
		\bfpage{261}--\blpage{274}
		(\byear{1995}).
		\doiurl{10.1007/bf02808204}
	\end{barticle}
	\endbibitem
	
	\bibitem{Haxell1995Induced}
	\begin{barticle}
		\bauthor{\bsnm{Haxell}, \binits{P.E.}},
		\bauthor{\bsnm{Kohayakawa}, \binits{Y.}},
		\bauthor{\bsnm{{\L}uczak}, \binits{T.}}:
		\batitle{The induced size-{Ramsey} number of cycles}.
		\bjtitle{Combin. Probab. Comput.}
		\bvolume{4}(\bissue{3}),
		\bfpage{217}--\blpage{239}
		(\byear{1995}).
		\doiurl{10.1017/s0963548300001619}
	\end{barticle}
	\endbibitem
	
	\bibitem{Javadi2019Size}
	\begin{barticle}
		\bauthor{\bsnm{Javadi}, \binits{R.}},
		\bauthor{\bsnm{Khoeini}, \binits{F.}},
		\bauthor{\bsnm{Omidi}, \binits{G.R.}},
		\bauthor{\bsnm{Pokrovskiy}, \binits{A.}}:
		\batitle{On the size-{Ramsey} number of cycles}.
		\bjtitle{Combin. Probab. Comput.}
		\bvolume{28}(\bissue{6}),
		\bfpage{871}--\blpage{880}
		(\byear{2019}).
		\doiurl{10.1017/s0963548319000221}
	\end{barticle}
	\endbibitem
	
	\bibitem{Javadi2018Question}
	\begin{barticle}
		\bauthor{\bsnm{Javadi}, \binits{R.}},
		\bauthor{\bsnm{Omidi}, \binits{G.}}:
		\batitle{On a question of {E}rd{\H{o}}s and {F}audree on the size {Ramsey}
			numbers}.
		\bjtitle{SIAM J. Discrete Math.}
		\bvolume{32}(\bissue{3}),
		\bfpage{2217}--\blpage{2228}
		(\byear{2018}).
		\doiurl{10.1137/17m1115022}
	\end{barticle}
	\endbibitem
	
	\bibitem{Lortz1998Size}
	\begin{barticle}
		\bauthor{\bsnm{Lortz}, \binits{R.}},
		\bauthor{\bsnm{Mengersen}, \binits{I.}}:
		\batitle{Size {Ramsey} results for paths versus stars}.
		\bjtitle{Australas. J. Combin.}
		\bvolume{18},
		\bfpage{3}--\blpage{12}
		(\byear{1998})
	\end{barticle}
	\endbibitem
	
	\bibitem{Lortz2021Size}
	\begin{barticle}
		\bauthor{\bsnm{Lortz}, \binits{R.}},
		\bauthor{\bsnm{Mengersen}, \binits{I.}}:
		\batitle{Size {Ramsey} results for the path of order three}.
		\bjtitle{Graphs Combin.}
		\bvolume{37}(\bissue{6}),
		\bfpage{2315}--\blpage{2331}
		(\byear{2021}).
		\doiurl{10.1007/s00373-021-02398-3}
	\end{barticle}
	\endbibitem
	
	\bibitem{Miralaei2019Size}
	\begin{barticle}
		\bauthor{\bsnm{Miralaei}, \binits{M.}},
		\bauthor{\bsnm{Omidi}, \binits{G.R.}},
		\bauthor{\bsnm{Shahsiah}, \binits{M.}}:
		\batitle{Size {Ramsey} numbers of stars versus cliques}.
		\bjtitle{J. Graph Theory}
		\bvolume{92}(\bissue{3}),
		\bfpage{275}--\blpage{286}
		(\byear{2019}).
		\doiurl{10.1002/jgt.22453}
	\end{barticle}
	\endbibitem
	
	\bibitem{rahadjeng2015connected}
	\begin{barticle}
		\bauthor{\bsnm{Rahadjeng}, \binits{B.}},
		\bauthor{\bsnm{Baskoro}, \binits{E.T.}},
		\bauthor{\bsnm{Assiyatun}, \binits{H.}}:
		\batitle{Connected size {Ramsey} numbers for matchings versus cycles or paths}.
		\bjtitle{Procedia Comput. Sci.}
		\bvolume{74},
		\bfpage{32}--\blpage{37}
		(\byear{2015}).
		\doiurl{10.1016/j.procs.2015.12.071}
	\end{barticle}
	\endbibitem
	
	\bibitem{Rahadjeng2016Connected}
	\begin{bchapter}
		\bauthor{\bsnm{Rahadjeng}, \binits{B.}},
		\bauthor{\bsnm{Baskoro}, \binits{E.T.}},
		\bauthor{\bsnm{Assiyatun}, \binits{H.}}:
		\bctitle{Connected size {Ramsey} numbers of matchings and stars}.
		In: \bbtitle{{AIP} Conference Proceedings},
		vol. \bseriesno{1707},
		pp. \bfpage{020015-1}--\blpage{020015-4}.
		\bpublisher{{AIP} Publishing {LLC}}
		(\byear{2016}).
		\doiurl{10.1063/1.4940816}
	\end{bchapter}
	\endbibitem
	
	\bibitem{Sheehan1984Class}
	\begin{bchapter}
		\bauthor{\bsnm{Sheehan}, \binits{J.}},
		\bauthor{\bsnm{Faudree}, \binits{R.J.}},
		\bauthor{\bsnm{Rousseau}, \binits{C.C.}}:
		\bctitle{A class of size {Ramsey} problems involving stars}.
		In: \beditor{\bsnm{Bollob{\'a}s}, \binits{B.}} (ed.)
		\bbtitle{Graph Theory and Combinatorics},
		pp. \bfpage{273}--\blpage{281}.
		\bpublisher{Cambridge},
		\blocation{London}
		(\byear{1984})
	\end{bchapter}
	\endbibitem
	
	\bibitem{Wang2022Connected}
	\begin{barticle}
		\bauthor{\bsnm{Wang}, \binits{S.}},
		\bauthor{\bsnm{Song}, \binits{R.}},
		\bauthor{\bsnm{Zhang}, \binits{Y.}},
		\bauthor{\bsnm{Zhang}, \binits{Y.}}:
		\batitle{Connected size {Ramsey} numbers of matchings versus a small path or cycle}. \bjtitle{arXiv preprint, arXiv:2205.03965}
		(\byear{2022}).
		{\href{https://arxiv.org/abs/2205.03965}{{arXiv:2205.03965}}}
	\end{barticle}
	\endbibitem
	
\end{thebibliography}
\end{document}